\documentclass{amsart}

\usepackage{amsmath}
\usepackage{amssymb}
\usepackage{color}
\input xy
\xyoption{all}
\usepackage[all,cmtip]{xy}
\usepackage{enumitem}
\usepackage{hyperref}
\usepackage{color}

\theoremstyle{definition}
\newtheorem{theorem}{Theorem}[section]

\theoremstyle{definition}
\newtheorem{thm}[theorem]{Theorem}
\newtheorem{defn}[theorem]{Definition}
\newtheorem{example}[theorem]{Example}
\newtheorem{c-example}[theorem]{Counter-example}
\newtheorem{lem}[theorem]{Lemma}
\newtheorem{cor}[theorem]{Corollary}
\newtheorem{prop}[theorem]{Proposition}
\theoremstyle{remark}
\newtheorem{rem}[theorem]{Remark}

\numberwithin{equation}{section}

\newcommand{\Cal}[1]{{\mathcal #1}}
\newcommand{\Sf}[1]{{\mathsf{#1}}}
\newcommand{\Par}[1]{\mathsf{Par}(\mathbb{#1})}
\newcommand{\paral}[1]{\ar@<0.3ex>[#1] \ar@<-0.3ex>[#1]}

\newcommand{\ParOrd}[1]{\mathsf{ParOrd}(\mathbb{#1})}
\newcommand{\Equiv}[1]{\mathsf{Eq}(\mathbb{#1})}
\newcommand{\Disc}[1]{\mathsf{Dis}(\mathbb{#1})}

\newcommand{\Stab}[1]{\mathsf{Stab}(\mathbb{#1})}
\newcommand{\parm}[2]{\langle{#1},{#2}\rangle}

\newcommand{\Id}[1]{\operatorname{Id}_{#1}}

\DeclareMathOperator{\Ob}{Ob}
\DeclareMathOperator{\Sub}{\mathsf{Sub}}
\DeclareMathOperator{\Hom}{hom}

\newdir{ >}{{}*!/-7pt/@{>}}

\begin{document}

\title[Pretorsion theories in lextensive categories]{Pretorsion theories in lextensive categories}

\author[F. Borceux]{Francis Borceux}
\address{Universit\'e catholique de Louvain, Institut de Recherche en Math\'ematique et Physique, 1348 Louvain-la-Neuve, Belgium}
\email{francis.borceux@uclouvain.be}

\author[F. Campanini]{Federico Campanini} \thanks{ The second named author is a postdoctoral researcher of the Fonds de la Recherche Scientifique - FNRS. He was also supported by Ministero dell'Istruzione, dell'Universit\`a e della Ricerca (PRIN ``Categories, Algebras: Ring-Theoretical and Homological Approaches (CARTHA)'')}
\address{Universit\'e catholique de Louvain, Institut de Recherche en Math\'ematique et Physique, 1348 Louvain-la-Neuve, Belgium}
\email{federico.campanini@uclouvain.be}

\author[M. Gran]{Marino Gran}
\address{Universit\'e catholique de Louvain, Institut de Recherche en Math\'ematique et Physique, 1348 Louvain-la-Neuve, Belgium}
\email{marino.gran@uclouvain.be}

\begin{abstract} 
We propose a construction of a stable category for any pretorsion theory in a lextensive category. We prove the universal property of the stable category, that extends previous results obtained for the stable category of internal preorders in a pretopos. Some examples are provided in the categories of topological spaces and of (small) categories.
\end{abstract}

\maketitle

\section*{Introduction}

Pretorsion theories were defined in \cite{FF, FFG} as ``non-pointed torsion theories", where the zero object and the zero morphisms are replaced by a class of “trivial objects” and an ideal of ``trivial morphisms”, respectively. This notion generalizes many concepts of torsion theory introduced and investigated by several authors in pointed and multi-pointed categories \cite{D, BG, CDT, GJ, JT}. Pretorsion theories appear in several different contexts, such as topological spaces and topological groups \cite{FFG}, internal preorders \cite{FF, FFG2}, categories \cite{Xarez}, preordered groups \cite{GM}, V-groups \cite{Michel}, crossed modules, etc.
An important example that inspired this new approach was studied in \cite{FF}, and it concerns the category $\mathsf{PreOrd}$ of preordered sets, that is, the category whose objects are the pairs $(A,\rho)$ where $A$ is a set and $\rho$ is a preorder on $A$. This category is equivalent to the category $\mathsf{AlexTop}$ of Aleksandrov-discrete topological spaces (recall that a topological space is called Aleksandrov-discrete if the intersection of any family of open subsets is an open subset). A pretorsion theory in $\mathsf{PreOrd}$ is given by the pair $(\mathsf{Eq}, \mathsf{ParOrd})$, where $\mathsf{Eq}$ and $\mathsf{ParOrd}$ are the full subcategories of $\mathsf{PreOrd}$ whose objects are equivalence relations and partially ordered sets, respectively. This example was extended in \cite{FFG2} to the category $\mathsf{PreOrd}(\mathbb C) $ of internal preorders in any \emph{exact} category $\mathbb{C}$ (in the sense of \cite{Barr}). Here the pretorsion theory is given by the pair $(\mathsf{Eq} (\mathbb C), \mathsf{ParOrd}  (\mathbb C))$, where $\mathsf{Eq} (\mathbb C)$ is the full subcategory of all symmetric internal preorders (i.e. the equivalence relations in $\mathbb C$) while $\mathsf{ParOrd}  (\mathbb C)$ is the full subcategory of all antisymmetric internal preorders (i.e. the partial orders in $\mathbb C$). If the category $\mathbb C$ is a \emph{pretopos}, it is possible to construct a pointed quotient category, called the ``{\em stable category}", with an interesting property: the pretorsion theory $(\mathsf{Eq}, \mathsf{ParOrd})$ in $\mathsf{PreOrd} (\mathbb C)$ becomes a ``classical" torsion theory in the stable category. This construction was presented in \cite{BCG, BCG2} as a natural generalization of the one proposed in \cite{FF} for the category $\mathsf{PreOrd}$ of preordered sets. It is worth noting that in the case $\mathbb{C}=\mathsf{Set}$, the stable category provides an example of a pointed category arising as a quotient category of a topological category (the category of Aleksandrov-discrete spaces), where short exact sequences naturally occur, as in the case of abelian categories. The pretorsion theory $(\mathsf{Eq}, \mathsf{ParOrd})$ gives rise to an interesting Galois theory, whose main properties have been explored in \cite{FFG2}.
The stable category presented in \cite{BCG, BCG2} has an important categorical interpretation: the canonical functor $\Sigma \colon \mathsf{PreOrd} (\mathbb C) \to \mathsf{Stab} $ (where $\mathsf{Stab}$ is the stable category) is universal among all (finite coproduct preserving) torsion theory functors $G\colon \mathsf{PreOrd} (\mathbb C) \to \mathbb{X}$, where $\mathbb{X}$ is a pointed category with coproducts equipped with a torsion theory (see \cite[Section~2]{BCG2} for more details). Informally speaking, the quotient functor $\Sigma \colon \mathsf{PreOrd} (\mathbb C) \to \mathsf{Stab} $ is the best possible finite coproduct preserving functor transforming the pretorsion theory $(\mathsf{Equiv}  (\mathbb C), \mathsf{ParOrd}  (\mathbb C))$ into a torsion theory.

 Considering this result, a question naturally arises: is it possible to associate a universal torsion theory with \emph{any} given pretorsion theory? In this paper we show that this is indeed possible whenever the base category is merely a \emph{lextensive} category \cite{CLW}, hence naturally including the example mentioned above and many others (see \cite{FF, FFG, FFG2, Xarez} and the last section of this article for more details).
Most of techniques contained in this paper are based on those of \cite{BCG, BCG2}. The main reason for this is that the category $\mathsf{PreOrd} (\mathbb C)$ of internal preorders in a pretopos $\mathbb C$ is lextensive, and the proofs in \cite{BCG, BCG2} essentially used the ``good'' properties relating coproducts and limits in the category $\mathsf{PreOrd} (\mathbb C)$. We now fully exploit this fact in the present work, making our results more general and the key steps in the proofs more clear.

\section{Preliminary notions}

A category $\mathbb{C}$ with finite sums (=coproducts) is {\em extensive} if it has pullbacks along coprojections in a sum and the following condition holds: in the commutative diagram, where the bottom row is a sum
$$
\xymatrix{
A \ar[r] \ar[d] & C \ar[d] & B \ar[l] \ar[d]\\
X_1 \ar[r] & X_1 +X_2 & X_2\ar[l]
}
$$
the top row is a sum if and only if the two squares are pullbacks \cite{CLW}. The property saying that the upper row of the diagram is a sum whenever the two squares are pullbacks is called ``universality of sums". 
 Recall that a sum of two objects $X_1$ and $X_2$ is \emph{disjoint} if the coprojections $i_1\colon X_1\to X_1 + X_2$ and $i_2\colon X_2\to X_1 + X_2$ are monomorphisms and the intersection $X_1\cap X_2$ is an initial object in the category $\Sub(X_1+ X_2)$ of subobjects of $X_1 +X_2$ (we shall use the ``intersection symbol" to denote the categorical product in $\Sub(X)$ and we shall call it intersection, as usual).
If $\mathbb{C}$ has finite sums and pullbacks along sum coprojections, extensivity is equivalent to the property of having disjoint and universal finite sums \cite[Proposition~2.14]{CLW}. Moreover, in an extensive category the initial object 0 is {\em strict}, that is, every morphism with codomain 0 is an isomorphism.

An extensive category with all finite limits is called {\em lextensive} \cite{CLW}. If $\mathbb{C}$ is a lextensive category, finite products distribute over the sums, that is there is a natural isomorphism $$
X_1\times(X_2+X_3)\cong (X_1\times X_2)+(X_1\times X_3)
$$
for every $X_1,X_2,X_3$ in $\mathbb{C}$.

\begin{example}
Any pretopos is a lextensive category, hence in particular the category $\Sf{Set}$ of sets. The category $\mathsf{PreOrd}(\mathbb C)$ of internal preorders and $\mathsf{Cat} (\mathbb C)$ of internal categories in a lextensive category $\mathbb C$ is again lextensive: in particular, this is the case for the category $\mathsf{Cat}$ of categories. Other examples of lextensive categories include the category $\mathsf{Top}$ of topological spaces, the dual category of the category of commutative rings and the category of affine schemes (see \cite{CLW}).
\end{example}

\begin{rem}
If $\mathbb{C}$ is a lextensive category and $X$ is an object in $\mathbb{C}$, then $\Sub(X)$ is preordered and the existence of finite limits implies that it is a bounded semi-lattice with meet $\cap$. Suppose we can write $X=A+A'$ for some $A,A'$ in $\mathbb{C}$. Since sums are disjoint, $A$ and $A'$ are subobjects of $X$ and $A+A'=X$ is the least upper bound of $A$ and $A'$ in $\Sub(X)$. Thus we can say (with a slight abuse of notation, since $\Sub(X)$ is not a lattice) that $A$ is a {\em complemented subobject of $X$} and $A'$ a complement of $A$ in $X$. Notice that coproduct decompositions $X=A+A'$ correspond to the morphisms $X\to 1+1$, in the following sense.
If we have a morphism $f:X\to 1+1$, then pulling back the coprojections $1\rightarrow 1+1 \leftarrow 1$ along $f$ we get a commutative diagram
$$
\xymatrix{
A \ar[r] \ar[d] & X \ar[d]^f & A' \ar[l] \ar[d]\\
1 \ar[r] & 1 + 1 & 1\ar[l]
}
$$
where the top row is a sum by the universality of sums.
On the other hand, a decomposition $X=A+A'$ together with the unique morphisms $A\to 1$ and $A'\to 1$ give us a unique morphism $f\colon X \to 1+1$.
\end{rem}

\section{Basics on pretorsion theories}
We recall the definition of pretorsion theory for general categories, as defined in \cite{FF, FFG}. Let $\mathbb{C}$ be an arbitrary category and consider a pair $(\Cal T, \Cal F)$ of two replete full subcategories of $\mathbb{C}$. Set also $\Cal Z:=\Cal T\cap\Cal F$ and call it \emph{the class of ``trivial objects"}. A morphism $f\colon A\to A'$ in $\mathbb{C}$ is $\Cal Z$-trivial if it factors through an object of $\Cal Z$. Notice that the class of trivial morphisms in $\mathbb{C}$ is an ideal of morphisms in the sense of Ehresmann \cite{Ehr}, and thus it is possible to consider the notions of $\mathcal Z$-kernel and of $\mathcal Z$-cokernel, defined by replacing, in the definition of kernel and cokernel, the ideal of zero morphisms with the ideal of trivial morphisms induced by the subcategory $\mathcal Z$. More precisely, we say that a morphism $\varepsilon\colon X\to A$ in $\mathbb{C} $ is a \emph{$\Cal Z$-kernel} of $f\colon A \to A'$ if $f\varepsilon$ is a $\Cal Z$-trivial morphism and, whenever $\lambda \colon Y\to A$ is a morphism in $\mathbb{C}$ and $f\lambda$ is $\Cal Z$-trivial there exists then a unique morphism $\lambda'\colon Y\to X$ in $\mathbb{C}$ such that $\lambda=\varepsilon\lambda'$.
The notion of \emph{$\Cal Z$-cokernel} is defined dually. A sequence $A\overset{f}{\to}B\overset{g}{\to}C$ is called a \emph{short $\Cal Z$-exact sequence} if $f$ is the $\Cal Z$-kernel of $g$ and $g$ is the $\Cal Z$-cokernel of $f$. We say that the pair $(\Cal T,\Cal F)$ is a \emph{pretorsion theory} in $\mathbb{C}$ if the following two properties are satisfied:
\begin{itemize}
\item
any morphism from an object $T\in\Cal T$ to an object $F\in\Cal F$ is $\Cal Z$-trivial;
\item
for every object $X$ of $\mathbb{C}$ there is a short $\Cal Z$-exact sequence
$$\xymatrix{ T_X \ar[r]^f &  X \ar[r]^g &  F_X}$$ with $T_X\in\Cal T$ and $F_X\in\Cal F$.
\end{itemize}
Recall from \cite{FFG} that given a pretorsion theory $(\Cal T, \Cal F)$ in a category $\mathbb{C}$, there are two functors:
\begin{itemize}
\item
a ``torsion functor" $T\colon \mathbb{C}\to \Cal T$ which is the right adjoint of the full embedding $E_T\colon \Cal T \to \mathbb{C}$ of the torsion subcategory $\Cal T$;
\item
a ``torsion-free functor" $F\colon \mathbb{C}\to \Cal F$ which is the left adjoint of the full embedding $E_F\colon \Cal F \to \mathbb{C}$ of the torsion-free subcategory $\Cal F$.
\end{itemize}
For every object $X \in \mathbb{C}$ there is a short $\Cal Z$-exact sequence
$$
\xymatrix{T(X)\ar[r]^-{\varepsilon_X} & X \ar[r]^-{\eta_X} & F(X)}
$$
where the monomorphism $\varepsilon_X$ is the $X$-component of the counit $\varepsilon$ of the adjunction
$$
\xymatrix{\Cal T \ar@/^5pt/[r]^{E_T} & \mathbb{C},\ar@/^5pt/[l]^{T}_\perp}
$$
while the epimorphism $\eta_X$ is the $X$-component of the unit $\eta$ of the adjunction
$$
\xymatrix{\mathbb{C} \ar@/^5pt/[r]^F & \Cal F \ar@/^5pt/[l]^{E_F}_\perp}
$$

The following lemma collects some basic properties of pretorsion theories we will use later on.
\begin{lem}\label{Lemma_1}
Let $(\Cal T,\Cal F)$ be a pretorsion theory in a category $\mathbb{C}$.
\begin{enumerate}
\item
The initial object 0 (if it exists) is in $\Cal T$, while the terminal object 1 (if it exists) is in $\Cal F$.
\item
If $\mathbb{C}$ has disjoint sums, then the initial object is in $\Cal T \cap \Cal F$.
\item
$\Cal T$ is closed under extremal quotients, i.e. if $q \colon X \rightarrow Y$ is an extremal epimorphism with $X$ in $\Cal T$, then $Y$ is in $\Cal T$. Dually $\Cal F$ is closed under extremal subobjects.
\item
If $\mathbb{C}$ is balanced, then $\Cal T$ is closed under quotients and $\Cal F$ is closed under subobjects.
\item
If $\Cal Z$ is closed under quotients [resp. subobjects], then $\Cal T$ is closed under quotients [resp. $\Cal F$ is closed under subobjects].
\item
Assume that $\mathbb{C}$ admits pullbacks along monomorphisms. Then $\Cal T$ is closed under subobjects if and only if for every monomorphism $n\colon X\to Y$ the left-hand square of the diagram
$$
\xymatrix{
T(X) \ar[r]^{\varepsilon_X}\ar[d]_{T(n)} & X \ar[r]^{\eta_X} \ar[d]^n & F(X) \ar[d]^{F(n)}\\
T(Y) \ar[r]_{\varepsilon_Y} & Y \ar[r]_{\eta_Y} & F(Y)
}
$$
is a pullback.
\item
$\Cal T$ [resp. $\Cal F$] is closed under all colimits [resp. limits] existing in $\mathbb{C}$.
\end{enumerate}
\end{lem}
\begin{proof}
$(1)$ Assume that $\mathbb{C}$ has an initial object. Then there is a $\Cal Z$-exact sequence $T(0)\to 0 \to F(0)$ with $T(0) \in \Cal T$ and $F(0)  \in \Cal F$. The first arrow is a monomorphism with codomain $0$, hence an isomorphism. Dually, one has the corresponding result for the terminal object $1$.

$(2)$ We only need to show that $0 \in \Cal F$. Let $T(0)\to 0 \to F(0)$ be the canonical short $\Cal Z$-exact sequence of $0$. Then $0 \to F(0)$ is an epimorphism, since it is a $\Cal Z$-cokernel. Thus the coprojections $F(0)\rightrightarrows F(0)+F(0)$ are equal, and therefore $F(0)\cap F(0)=F(0)$ as subobjects of $F(0)+F(0)$. But sums are disjoint, so we have $0= F(0)\cap F(0)= F(0) \in \Cal F$.

$(3)$ It suffices to notice that the proof given in \cite[Corollary~5.3]{FFG} works without any assumption on $\Cal Z$.

$(4)$ Assume that $\mathbb{C}$ is balanced and let $n\colon N\to X$ be a monomorphism with $X \in \Cal F$. Then we have a commutative diagram
$$
\xymatrix{
T(N) \ar[r]^{\varepsilon_N}\ar[d]_{T(n)} & N \ar[r]^{\eta_N} \ar[d]^n & F(N) \ar[d]^{F(n)}\\
T(X) \ar[r]_{\varepsilon_X} & X \ar@{=}[r]_{\eta_X} & X
}
$$
where the two rows are $\Cal Z$-exact sequences and $\eta_N$ is both an epimorphism (because it is a $\Cal Z$-cokernel) and a monomorphism (because so is $n$). Under our assumption $\eta_N$ is then an isomorphism and $N \in \Cal F$. The statement for $\Cal T$ can be proved in a similar way.

$(5)$ Assume that $\Cal Z$ is closed under subobjects and let $n\colon N\to X$ be a monomorphism with $X \in \Cal F$. We get the same diagram as above, where $T(X)\in \Cal Z$ and $T(n)$ is a monomorphism because $n\circ\varepsilon_N$ is so. By hypothesis $T(N) \in \Cal Z$ and thus $\eta_N$ is an isomorphism by \cite[Lemma~2.4]{FFG}. The statement for $\Cal T$ can be proved in a similar way.

$(6)$ Let us first prove the ``if-part". For this, suppose that $Y \in \Cal T$ and $n\colon X\to Y$ is a monomorphism. Then $\varepsilon_Y$ is an isomorphism, hence so is $\varepsilon_X$. Thus $X\cong T(X) \in \Cal T$, that is $\Cal T$ is closed under subobjects.

Now assume that $\Cal T$ is closed under subobjects and consider the diagram
$$
\xymatrix@!=40pt{
P \ar@/^2pc/[rrd]^b \ar@/_2pc/[rdd]^a \ar@/^1pc/[rd]^{T(b)}& & \\
 & T(X) \ar@/^1pc/[ul]^{\sigma} \ar[d]_{T(n)} \ar[r]^{\varepsilon_X} & X \ar[d]^n\\
 & T(Y) \ar[r]_{\varepsilon_Y} & Y\\
}
$$
where the outer quadrangle is a pullback, $\sigma$ is the unique morphism induced by the universal property of the pullback, and $T(b)$ is the (unique) morphism such that $\varepsilon_X T(b)=b$ (observe that $P \in \Cal T$ because $a$ is a monomorphism and $T(Y)\in \Cal T$). It is easy to see that all arrows are monomorphisms and that $\sigma$ is an isomorphism.

$(7)$ It follows from the fact that $\Cal T$ [resp. $\Cal F$] is a coreflective [resp. reflective] full subcategory of $\mathbb{C}$.
\end{proof}
\begin{rem}
The property $(6)$ above generalizes the following well-known result for torsion theories in abelian categories: a torsion theory $(\Cal T, \Cal F)$ is hereditary if and only if for every monomorphism $X\rightarrow Y$ we have $T(X)=T(Y)\cap X$ (see also \cite{BG} for the more general case of homological categories).
\end{rem}
\begin{example}
There are only three pretorsion theories in $\Sf{Set}$, namely $(\{\emptyset\},\Sf{Set})$, $(\Sf{Set}, \Sf{Set})$ and $(\Sf{Set}, \Cal F_0)$, where $\Cal F_0$ is the class of sets with at most one element.

Indeed, the category of sets has only one proper full coreflective subcategory, that is, the one whose unique object is the empty set, and only one proper full reflective subcategory closed under subobjects, namely $\Cal F_0$.

\end{example}

In the last part of this section we aim at identifying some properties of pretorsion theories that hold in any lextensive category $\mathbb{C}$.

\begin{rem}
It is worth noting that most of the results we are going to present until the end of this section hold in any extensive category $\mathbb{C}$, because the only limits involved are pullbacks along injections. Nevertheless, since we are interested in lextensive categories, for ease of readability and in order to avoid confusion, we prefer to always assume that $\mathbb{C}$ has all finite limits.
\end{rem}

\begin{prop}
Let $(\Cal T, \Cal F)$ be a pretorsion theory in a lextensive category $\mathbb{C}$. Then $\Cal T$ is closed under coproducts, that is, if $X,Y \in \Cal T$ then $X+Y \in \Cal T$.
\end{prop}

\begin{proof}
This result is a particular case of Lemma~\ref{Lemma_1}(7).
\end{proof}

\begin{prop}\label{F_closed_complemented_subobjects}
Let $(\Cal T, \Cal F)$ be a pretorsion theory in a lextensive category $\mathbb{C}$. Then $\Cal F$ is closed under complemented subobjects, that is, if $X+Y \in \Cal F$ then $X,Y \in \Cal F$.
\end{prop}

\begin{proof}
Let $X, Y \in \mathbb{C}$ be two objects such that $X+Y \in \Cal F$ and consider the commutative diagram 
$$
\xymatrix@=35pt{
X \ar@{>->}[r]^-{\iota_1} \ar[d]_{\eta_X} & X+Y \ar@<0.5ex>[d]^\sigma & Y \ar@{>->}[l]_-{\iota_2} \ar[d]^{\eta_Y}\\
F(X) \ar[ru]^-{F(\iota_1)} \ar@{>->}[r]_-{\iota_1'}& F(X)+F(Y)\ar@<0.5ex>[u]^\tau & F(Y) \ar@{>->}[l]^-{\iota_2'} \ar[lu]_{F(\iota_2)}
}
$$
where $\tau$ and $\sigma$ are the morphisms induced by the coproducts on the bottom and top row respectively (notice that a priori $X+Y$ is the coproduct of $F(X)$ and $F(Y)$ in $\Cal F$ but not necessarily in $\mathbb{C}$). A straightforward computation shows that $\sigma$ is the inverse of $\tau$. Moreover, the two squares
$$
\xymatrix{
X \ar@{>->}[r]^-{\iota_1} \ar[d]_{\eta_X} & X+Y \ar[d]^\sigma & Y \ar@{>->}[l]_-{\iota_1} \ar[d]^{\eta_Y}\\
F(X) \ar@{>->}[r]_-{\iota_1'} & F(X)+F(Y) & F(Y) \ar@{>->}[l]^-{\iota_2'} 
}
$$
are pullbacks by extensivity, and so $\eta_X$ and $\eta_Y$ are isomorphisms. It follows that $X$ and $Y$ are in $\Cal F$, as desired.
\end{proof}

\begin{lem}\label{lemma_coproduct}
Let $(\Cal T, \Cal F)$ be a pretorsion theory in a lextensive category $\mathbb{C}$. Let $g_1\colon X_1\to Y$ and $g_2 \colon X_2 \to Y$ be two trivial morphisms. Then the induced morphism $g\colon X_1+X_2\to Y$ is trivial provided one of the following conditions hold:
\begin{enumerate}
\item[(a)]
$g$ has a $\Cal Z$-kernel;
\item[(b)]
$\Cal Z$ is closed under coproducts.
\end{enumerate}
\end{lem}

\begin{proof}
First assume that $g$ has a $\Cal Z$-kernel. Notice that $\Id {X_\ell}\colon X_\ell \to X_\ell$ is the $\Cal Z$-kernel of $g_\ell$ for $\ell=1,2$. Let $k\colon K\to X_1+X_2$ be the $\Cal Z$-kernel of $g$ and let $\iota_\ell\colon X_\ell\rightarrowtail X_1+X_2$ $(\ell=1,2)$ be the coprojections, which factor through $k$ because $g\cdot \iota_\ell=g_\ell$ are trivial. Thus we have a commutative diagram
$$
\xymatrix{
X_1 \ar[r]^-{j_1} \ar@{=}[d] & K \ar[d]^{k} & X_2 \ar[l]_-{j_2} \ar@{=}[d]\\
X_1 \ar@{ >->}[r]^-{\iota_1} \ar[dr]_{g_1} & X_1+X_2 \ar[d]^{g} &  X_2 \ar@{ >->}[l]_-{\iota_2} \ar[dl]^{g_2}\\
 & Y &
}
$$
and the monomorphism $k$ is actually an isomorphism, since the pair ($\iota_1$, $\iota_2$) is jointly extremal epimorphic.
By \cite[Lemma~2.4(b)]{FFG} $g$ is trivial.

Now assume that $\Cal Z$ is closed under coproducts. If $g_\ell$ factors through $Z_\ell \in \Cal Z$ ($\ell=1,2$), then $g$ factors through $Z_1+Z_2$.
\end{proof}

\begin{lem}\label{magenta}
Let $(\Cal T, \Cal F)$ be a pretorsion theory in a lextensive category $\mathbb{C}$. Then the following conditions are equivalent:
\begin{enumerate}
\item[(i)]
$\Cal Z$ is closed in $\mathcal C$ under complemented subobjects, that is, if $X+Y \in \Cal Z$ then $X,Y \in \Cal Z$;
\item[(ii)]
the ideal of trivial morphisms is stable under pullbacks along coprojections, that is, for every pullback diagram of the form
$$
\xymatrix{
W \ar[r]^{\tilde{\beta}} \ar[d]_{\tilde{g}} & X \ar[d]^g\\
B \ar@{ >->}[r]_\beta & Y 
}
$$
where $B$ is a complemented subobject of $Y$, if $g$ is trivial then $\tilde{g}$ is trivial;
\item[(iii)]
if the composite morphism $\xymatrix{ W\ar[r]^g & {B\, } \ar@{>->}[r]^\beta & Y}$ is trivial and $\beta$ is a coprojection, then $g$ is trivial.
\end{enumerate}
\end{lem}

\begin{proof}$(i)\Rightarrow (ii)$
We can write $g \tilde{\beta}=hk$ where the domain of $h$ is a trivial object $Z \in \Cal Z$. We then get a commutative diagram
$$
\xymatrix{
W \ar[r]^{\tilde{\beta}} \ar[d]^{\tilde{k}} \ar@/^-1.5pc/[dd]_{\tilde{g}} & X \ar[d]^k \ar@/^1.5pc/[dd]^{g}\\
Z_1 \ar@{ >->}[r] \ar[d]^{\tilde{h}} & Z \ar[d]^h &\\
B  \ar@{ >->}[r]_{\beta} & Y 
}
$$
where the squares are pullbacks and $\tilde{h}\tilde{k}=\tilde{g}$. The assumption implies that $Z_1 \in \Cal Z$: it follows that $\tilde{g}$ is a trivial morphism.

$(ii)\Rightarrow (i)$ if $X+Y \in \Cal Z$, then the coprojection $\iota_1\colon X \rightarrowtail X+Y$ is a trivial morphism. Applying $(ii)$ to the pullback of the monomorphism $\iota_1$ with itself we get that $\Id X$ is a trivial morphism, hence $X \in \Cal Z$.

$(i)\Rightarrow (iii)$ We can write $g \beta=hk$ where the domain of $h$ is a trivial object $Z \in \Cal Z$. Thus we have a commutative diagram
$$
\xymatrix{
W \ar[dr]_g  \ar@/^2pc/[rr]^k \ar@{..>}[r] &  Z' \ar@{>->}[r]\ar[d] & Z\ar[d]^h \\
& B \ar@{>->}[r]_\beta & Y
}
$$
where the square is a pullback. The morphism $g$ is then trivial, since it factors through the trivial object $Z'$.

$(iii)\Rightarrow (i)$ if $X+Y \in \Cal Z$, then the morphism $\xymatrix{X \ar@{=}[r] & X \ar@{>->}[r]^-{\iota_1} & X+Y}$ is trivial, hence so is $\Id X$. It follows that $X \in \Cal Z$.\\
\end{proof}

\begin{rem}\label{magenta_remark}
It is clear that Lemma~\ref{magenta} also holds when ``coproduct coprojections'' are replaced by ``monomorphisms''. Thus, using the same arguments as above, we get that the following conditions are equivalent.
\begin{enumerate}
\item[(i)] $\Cal Z$ is closed under subobjects;
\item[(ii)] the ideal of trivial morphisms is stable under pullbacks along monomorphisms;
\item[(iii)] if a composite morphism $m g$ is trivial and $m$ is a monomorphism, then $g$ is trivial.
\end{enumerate}
We can also consider another distinguished class of monomorphisms, namely those monomorphisms that are $\Cal Z$-kernels. It is easy to check that if a composite morphism $k g$ is trivial and $k$ is a $\Cal Z$-kernel, then $g$ is trivial. In particular, we get the following result.
\end{rem}

\begin{prop}
The ideal of trivial morphisms is stable under pullbacks along $\Cal Z$-kernels.
\end{prop}

\section{Partial morphisms and the stable category}\label{section_STABLE}

From now on $\mathbb{C}$ will denote a lextensive category and $(\Cal T, \Cal F)$ will be a pretorsion theory in $\mathbb{C}$ such that the class of trivial objects $\Cal Z:=\Cal T \cap \Cal F$ (and hence the class of trivial morphisms) satisfies the following properties:

\begin{itemize}
\item[-]
Property $(1)$: $\Cal Z$ is closed in $\mathbb{C}$ under binary coproducts;

\item[-]
Property $(2)$: $\Cal Z$ is closed in $\mathbb{C}$ under complemented subobjects.
\end{itemize}

 Let $\mathbb{C}$ be a lextensive category. We define the \emph{category $\Par C$ of partial morphisms} in $\mathbb{C}$. The concept of partial morphisms has been largely studied in literature. Here, we define what is called ``the category of $\Cal M$-partial morphisms" in \cite{RR}, where $\Cal M$ is the class of complemented subobjects (see \cite{RR} and the references therein for more details about partial maps).
 The objects of $\Par C$ are the same as those in $\mathbb C$. A morphism $X\to Y$ in the category $\Par C$ is an equivalence class of pairs $(\alpha, f)$ where $\alpha \colon A\rightarrowtail X$ is a complemented subobject of $X$, $f$ is a morphism in $\mathbb{C}$ and two pairs $(\alpha, f)$ and $(\alpha',f')$ (with $\alpha'\colon A'\rightarrowtail X$ and $f'\colon A'\to Y$) are equivalent if there exists a (unique) isomorphism $u$ such that $\alpha=\alpha' u$ and $f=f' u$.
 
A morphism $(\alpha, f)\colon X\to Y$ in $\Par C$ is displayed as 
$$
\xymatrix{
& A \ar@{ >->}[dl]_{\alpha} \ar[dr]^{f} & \\ X \ar@{.>}[rr]_{(\alpha, f)}& & Y
}
$$

The composite $(\beta, g) \circ (\alpha, f)$ in $\Par C$ of two morphisms $(\alpha, f) \colon X \to Y$ and $(\beta, g) \colon Y\to Z$ in $\Par C$ is defined by the external part of the diagram
$$
\xymatrix@=15pt{
 & & A'  \ar[rd]^{f'} \ar@{ >->}[dl]_{\alpha'} & & \\
& A \ar@{ >->}[dl]_{\alpha} \ar[dr]^{f}& & B \ar@{ >->}[dl]_{\beta} \ar[dr]^{g}& \\
X \ar@{.>}[rr]_{(\alpha,f)} & & Y \ar@{.>}[rr]_{(\beta,g)} & &  Z
}
$$
where the upper part is a pullback: in other words,
$$
(\beta, g) \circ (\alpha, f)= (\alpha \alpha', g f').
$$
Remark that the composition is well-defined, since $A'$ is a complemented subobject of both $A$ and $X$. By the elementary properties of pullbacks one sees that this composition is associative, and the identity in $\Par C$ of an object $X$ is given by the arrow
$$
\xymatrix{
& X \ar@{=}[dl]_{1} \ar@{=}[dr]^{1} & \\ X \ar@{.>}[rr]_1& & X
}
$$

\begin{rem}
 Notice that given a partial map $(\alpha, f)\colon X\to Y$, the subobject $\alpha \colon A \rightarrow X$ can only be determined up to isomorphism, being the representative of a class of monomorphisms $A \to X$ in $\mathbb{C}$. Nevertheless, it is easy to prove that the composition is independent of the choice of representatives.
\end{rem}

There is a functor
$$
I \colon \mathbb{C} \longrightarrow \Par C
$$
which is the identity on objects and such that, for any morphism $f \colon X\to Y$ in $\mathbb C$, $I(f) \colon X\to Y$ in $\Par C$ is defined by
$$
\xymatrix{
& X \ar@{=}[dl]_{1} \ar[dr]^{f} & \\
X \ar@{.>}[rr]_{I(f)}& & Y
}
$$

\noindent { \bf Convention.}
 From now on, an arrow $\xymatrix{\ar@{ >->}[r] &  }$ will always denote a complemented subobject, whereas an arrow $\xymatrix{ \ar[r] & }$ will be a morphism in $\mathbb{C}$.\\

The category $\Par C$ is pointed with zero object given by the initial object $0$ of $\mathbb{C}$. Indeed, since the initial object 0 of $\mathbb{C}$ is strict, it is easy to see that for any object $X$ the unique morphisms in $\Par C$ from $X$ to $0$ and from $0$ to $X$ are 
 $$\xymatrix{
\ar@{}[drrrrrrr]|{\mbox{and}} & 0 \ar@{ >->}[dl]_{} \ar@{=}[dr]^{} & & & & & 0 \ar[dr]_{} \ar@{=}[dl]_{} & \\
 X \ar@{.>}[rr]_{\omega_X}& & 0 & & & 0 \ar@{.>}[rr]_{\alpha_X}& & X,
}
$$
respectively.

 Starting from the category $\Par C$ of partial morphisms of $\mathbb{C}$ we are now going to define a quotient category of $\Par C$, that we shall call the \emph{stable category}.

\begin{defn}
A \emph{congruence diagram} in $\Par C$ is a diagram in $\mathbb{C}$ of the form
\begin{equation}\label{C-diagram}  
\xymatrix@=35pt{
C_1^c \ar@{ >->}[rr]^{\gamma_1^c } & & A_1  \ar@{ >->}[ld]^{\alpha_1} \ar[dr]^{f_1}  & & \\
C \ar@{ >->}[drr]_{\gamma_2} \ar@{ >->}[rru]^{\gamma_1 } \ar@{ >->}[r]_{\gamma_0} & X   &   & Y \\
C_2^c \ar@{ >->}[rr]_{\gamma_2^c}  & & A_2 \ar@{ >->}[ul]_{\alpha_2} \ar[ru]_{f_2} & &
 }
\end{equation}
where:
\begin{itemize}
\item[-]
the two triangles commute;
\item[-]
$A_1=C+C_1^c$ and $A_2=C+C_2^c$;
\item[-]
$f_1  \gamma_1 = f_2 \gamma_2$;
\item[-]
$f_1 {\gamma_1}^c$ and $f_2 {\gamma_2}^c$ are both trivial morphisms in $\mathbb C$.
\end{itemize}

Given two parallel morphisms $(\alpha_1, f_1)$ and $(\alpha_2, f_2)$ in $\Par C$, depicted as
$$
\xymatrix{
\ar@{}[drrrrrrr]|{\mbox{and}} & A_1 \ar@{ >->}[dl]_{\alpha_1} \ar[dr]^{f_1} & & & & & A_2 \ar@{ >->}[dl]_{\alpha_2} \ar[dr]^{f_2} & \\
 X \ar@{.>}[rr]_{} & & Y & & & X \ar@{.>}[rr]_{}& & Y,
 }
$$
one says that they are \emph{equivalent}, and writes $(\alpha_1, f_1) \sim (\alpha_2, f_2)$, if there is a congruence diagram of the form \eqref{C-diagram}.
\end{defn}

\begin{prop}\label{cong} 
 The relation $\sim$ defined above is a congruence on $\Par C$, i.e. an equivalence relation which is also compatible with the composition in $\Par C$.
\end{prop}

\begin{proof}
 It is clear that the relation $\sim$ is symmetric. To see that it is reflexive it suffices to choose $C=A_1$ in diagram \eqref{C-diagram}, so that $C_1^c=0$ and $f_1 \gamma_1^c$ is trivial because $0 \in \Cal Z$.
 To see that $\sim$ is transitive consider three parallel morphisms in $\Par C$ such that $(\alpha_1, f_1) \sim (\alpha_2, f_2)$ and $(\alpha_2, f_2) \sim (\alpha_3, f_3)$. There are then a congruence diagram \eqref{C-diagram} and a congruence diagram 

\begin{equation}\label{C2-diagram}  
\xymatrix@=35pt{
D_2^c \ar@{ >->}[rr]^{\delta_2^c} & & A_2 \ar@{ >->}[ld]^{\alpha_2} \ar[dr]^{f_2}  & & \\
 D \ar@{ >->}[drr]_{\delta_3} \ar@{ >->}[rru]^{\delta_2 } \ar@{ >->}[r]_{\delta_0} & X   & & Y \\
 D_3^c \ar@{ >->}[rr]_{\delta_3^c}  & & A_3 \ar@{ >->}[ul]_{\alpha_3} \ar[ru]_{f_3} & &
}
\end{equation}
This implies at once  that $f_1$, $f_2$ and $f_3$ coincide on $E:=C\cap D$. It remains to prove that $f_1$ is trivial on the complement $E_1^c=(C\cap D_2^c)+C_1^c$ of $E$ in $A_1$ and $f_3$ is trivial on the complement $E_3^c=(C \cap D_2^c)+D_3^c$ of $E$ in $A_3$.

We already know that $f_1$ is trivial on $C_1^c$, hence on $(C_1^c \cap D)$. Moreover, $f_1$ coincides with $f_2$ on $C$ and $f_2$ is trivial on $D_2^c$, therefore $f_1$ is trivial on $C \cap D_2^c$. Thus $f_1$ is trivial on $E_1^c$ by Property~$(1)$. Similarly for $f_3$, and this proves the transitivity of the relation $\sim$.

Let us then show that the equivalence relation $\sim$ is compatible with the composition. Let us first prove that, given a morphism
$$
\xymatrix{
& B \ar@{ >->}[dl]_{\beta} \ar[dr]^{g} & \\ Y \ar@{.>}[rr]& & Z}
$$
and two parallel morphisms $( \alpha_1, f_1) \colon X \rightarrow Y$ and $( \alpha_2, f_2) \colon X \rightarrow Y$ in $\Par C$ (as above) such that $( \alpha_1, f_1) \sim ( \alpha_2 , f_2)$, then
$$
( \beta, g) \circ ( \alpha_1, f_1) \sim ( \beta, g) \circ ( \alpha_2, f_2).
$$
By using the same notations as above for the congruence diagram \eqref{C-diagram} making $( \alpha_1, f_1)$ and $(\alpha_2 , f_2)$ equivalent, one can consider the composition diagram
$$
\xymatrix@=20pt{
C_1^c\cap \tilde{A_1} \ar@{ >->}[rr] & &\tilde{A_1}  \ar@{ >->}[d] \ar[ddrr]^{\tilde{f_1}} & & & \\
& & A_1\ar[dr]^{f_1} \ar@{ >->}[dl]_{\alpha_1}  & & & \\
C \cap \tilde{A_1} \cap \tilde{A_2} \ar@{ >->}[r] \ar@{ >->}[rruu]  \ar@{ >->}[ddrr] & X \ar@<-1ex>@{.>}[rr] \ar@<1ex>@{.>}[rr] & &Y & B \ar[r]^g \ar@{ >->}[l]_\beta& Z \\
& & A_2 \ar[ru]_{f_2}  \ar@{ >->}[lu]^{\alpha_2} & & & \\
C_2^c \cap \tilde{A_2}   \ar@{ >->}[rr]  & & \tilde{A_2} \ar@{ >->}[u] \ar[rruu]_{\tilde{f_2}} & & &
}
$$
where the two right-hand quadrangles are pullbacks by construction (so that $ \tilde{A_1} = f_1^{-1}(B)$ and $ \tilde{A_2} = f_2^{-1}(B)$).

Notice that
$
C\cap \tilde{A_1}=(f_1 \gamma_1)^{-1}(B)$ and $C \cap \tilde{A_2}=(f_2 \gamma_2)^{-1}(B) 
$, thus the equality 
$f_1 \gamma_1 = f_2 \gamma_2$ implies that
$$
C \cap \tilde{A_1}  = C \cap \tilde{A_2} = C \cap \tilde{A_1} \cap \tilde{A_2}.
$$ 
From this we get that the two left-hand quadrangles commute, that $C_i^c \cap \tilde{A_i}$ is the complement of $C \cap \tilde{A_1} \cap \tilde{A_2}$ in $\tilde{A_i}$ for $i=1,2$, and that $f_1$ and $f_2$ coincide on $C \cap \tilde{A_1} \cap {\tilde{A_2}}$.
Moreover, the restriction of $f_1$ to $C_1^c \cap \tilde{A_1}$ is a trivial morphism, hence so is the restriction of $\tilde{f_1}$ to $C_1^c \cap \tilde{A _1}$ by Property~$(2)$. We then conclude that also the restriction of $g \tilde{f_1}$ to $C_1^c \cap \tilde{A _1}$ is trivial. Similarly the restriction of $g \tilde{f_2}$ to $C_2^c \cap \tilde{A _2}$ is trivial, hence $ ( \beta, g) \circ ( \alpha_1, f_1) \sim ( \beta, g) \circ ( \alpha_2, f_2)$.

Next consider 
$$
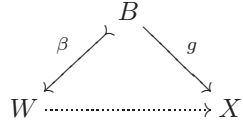
\xymatrix{
& B \ar@{ >->}[dl]_{\beta} \ar[dr]^{g} & \\ W \ar@{.>}[rr]_{}& & X
}
$$
and two morphisms $( \alpha_1, f_1) \colon X \rightarrow Y$ and $( \alpha_2, f_2) \colon X \rightarrow Y$ in $\Par C$ such that $( \alpha_1, f_1)  \sim  ( \alpha_2, f_2)$ with congruence diagram \eqref{C-diagram}.
Consider the diagram 
$$
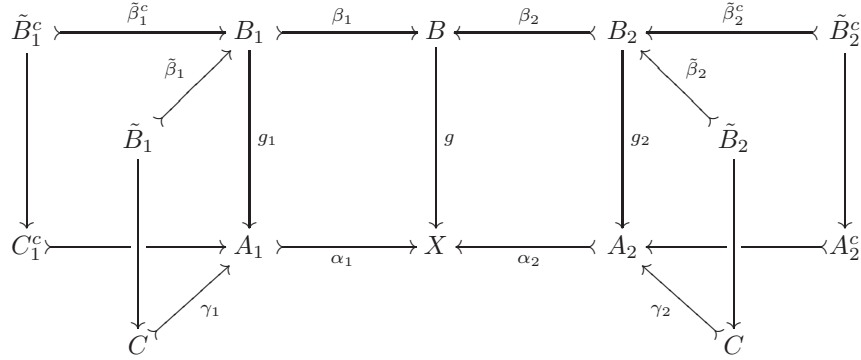
\xymatrix{
\tilde{B}_1^c \ar@{ >->}[rr]^{\tilde{\beta}_1^c} \ar[dd]& & B_1 \ar@{ >->}[rr]^{\beta_1} \ar[dd]^{g_1} & & B \ar[dd]^g & & B_2 \ar@{ >->}[ll]_{\beta_2} \ar[dd]^{g_2} & & \tilde{B}_2^c \ar@{ >->}[ll]_{\tilde{\beta}_2^c} \ar[dd] \\
 & \tilde{B}_1 \ar@{>->}[ru] ^{\tilde{\beta}_1}\ar[dd] & & & & & &  \tilde{B}_2 \ar@{>->}[lu]_{\tilde{\beta}_2}\ar[dd] \\
C_1^c \ar@{>-}[r] &\ar@{->}[r] & A_1 \ar@{ >->}[rr]_{\alpha_1} & & X & & A_2\ar@{ >->}[ll]^{\alpha_2} & \ar@{->}[l] & A_2^c \ar@{>-}[l] \\
 & C\ar@{>->}[ru]_{\gamma_1} & & & & & & C \ar@{>->}[lu]^{\gamma_2}
}
$$
where all the squares are pullbacks, so that by extensivity $\tilde{B}_i^c$ is the complement of $\tilde{B}_i$ in $B_i$ for $i \in \{1,2\}$ and, moreover, $\tilde{B}_1=\tilde{B}_2=g^{-1}(C)$.
It is then easy to check that the diagram
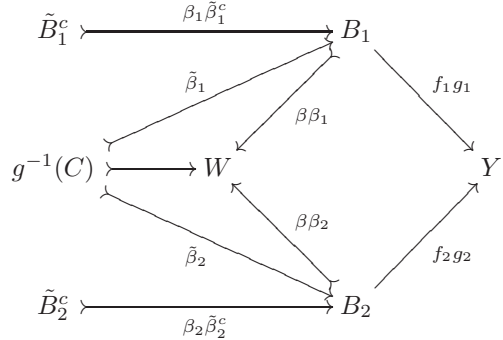
\[
\xymatrix@=35pt{
\tilde{B}_1^c \ar@{ >->}[rr]^{\beta_1\tilde{\beta}_1^c} & & B_1\ar@{ >->}[ld]^{\beta \beta_1} \ar[dr]^{f_1g_1}  & & \\
g^{-1}(C) \ar@{ >->}[drr]_{\tilde{\beta}_2} \ar@{ >->}[rru]^{\tilde{\beta}_1} \ar@{ >->}[r] & W   & & Y \\
\tilde{B}_2^c \ar@{ >->}[rr]_{\beta_2\tilde{\beta}_2^c}  & & B_2 \ar@{ >->}[ul]_{\beta \beta_2} \ar[ru]_{f_2g_2} & &
}
\]
(where the arrow $\xymatrix{ g^{-1}(C) \ar@{ >->}[r] &W}$ is $\beta \beta_1\tilde{\beta}_1=\beta \beta_2\tilde{\beta}_2$) is a congruence diagram. Therefore $( \alpha_1, f_1) \circ ( \beta, g)  \sim ( \alpha_2, f_2) \circ  ( \beta, g)$, as desired.
\end{proof}

\begin{defn}
We denote by $\Stab C$ the quotient category of $\Par C$ by the congruence $\sim$ in Proposition \ref{cong}. If we write $\pi \colon \Par C \rightarrow \Stab C$ for the quotient functor, we also have a functor
$$
\Sigma = \pi \circ I \colon \mathbb{C} \rightarrow \Stab C
$$
obtained by precomposing $\pi$ with the inclusion functor $I \colon  \mathbb{C}\rightarrow \Par C$.
\end{defn}
 
{ \bf Convention.} From now on we shall write $\parm \alpha f \colon X \rightarrow Y$ for the morphism $\pi (\alpha, f) \colon X \rightarrow Y$ in $\Stab C$ (which is a congruence class of morphisms in $\Par C$ by definition). For ease of readability, if $f\colon X\to Y$ is a morphism in $\mathbb{C}$, we shall often keep writing $f$ also for the morphism $\Sigma(f)=\parm {1_X} f$ in $\Stab C$.
 
\begin{rem}\label{zero}
The category $\Stab C$ is pointed with zero object given by the initial object $0$ of $\mathbb C$. Moreover, the following results can be proved by using the same arguments as in \cite{BCG}.
\begin{enumerate}
\item
For an object $X \in \mathbb C$ the following conditions are equivalent:
\begin{enumerate}
\item
$X$ is a trivial object;
\item
$\Sigma (X) \cong 0$.
\end{enumerate}

\item
For a morphism $f \colon X \rightarrow Y$ in $\mathbb C$ the following conditions are equivalent:
\begin{enumerate}
\item
$f$ is a trivial morphism;
\item
$\Sigma (f)$ is a zero morphism in $\Stab C$.
\end{enumerate}
\item
For a morphism $\xymatrix@=30pt{X  \ar[r]^{\parm \alpha f} &  Y }$ in $\Stab C$ the following conditions are equivalent:
\begin{enumerate}
\item
$\parm \alpha f = 0$;
\item
$f$ is a trivial morphism in $\mathbb C$.
\end{enumerate}
\end{enumerate}
\end{rem}

\section{Preservation properties of $\Sigma$}

\begin{prop}\label{coproducts-preserved}
The functor $\Sigma \colon \mathbb C \rightarrow \Stab C$ preserves finite coproducts.
\end{prop}

\begin{proof}
We already know that $\Sigma$ preserves the initial object, then it will suffice to prove that $\Sigma$ also preserves binary coproducts.
Consider then the coproduct
$$
\xymatrix{X\ar@{ >->}[r]^-{\iota_1} & X+Y & Y \ar@{ >->}[l]_-{\iota_2}}
$$
of two objects $X$ and $Y$ in $\mathbb C$ and two morphisms in $\Par C$
$$
\xymatrix{
\ar@{}[drrrrrrr]|{\mbox{and}} & A \ar@{>->}[dl]_{\alpha} \ar[dr]^{f} & & & & & B \ar@{>->}[dl]_{\beta} \ar[dr]^{g} & \\
 X \ar@{.>}[rr]_{(\alpha, f)} & & Z & & & Y \ar@{.>}[rr]_{(\beta, g)} & & Z,
 }
$$
We have to prove that there is a unique morphism $\parm \gamma h \colon X+Y \to Z$ in $\Stab C$ such that $\parm \gamma h \cdot \iota_1= \parm \alpha f$ and $\parm \gamma h \cdot \iota_2= \parm \beta g$ in $\Stab C$. First observe that $A+B$ is a complemented subobject of $X+Y$ whose complement is $A^c+B^c$, where $A^c$ and $B^c$ are the complements of $A$ and $B$ in $X$ and $Y$ respectively. Then, the following commutative diagram in $\mathbb{C}$
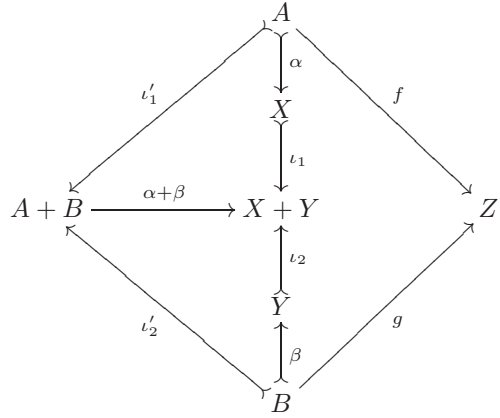
$$
\xymatrix{
  & &  A \ar@{ >->}[d]^{\alpha} \ar[ddrr]^f \ar@{>->}[ddll]_{\iota_1'}&  & & \\
  & &  X  \ar@{>->}[d]^{\iota_1} & & \\
  A+B \ar[rr]^{\alpha + \beta}  &
  & X+Y & & Z \\
 & & Y \ar@{>->}[u]_{\iota_2} & & & \\
 & & B \ar@{ >->}[u]_{\beta} \ar[rruu]_g \ar@{>->}[uull]^{\iota_2'} & & &
 }
$$
induces a unique morphism $h \colon A+B \to Z$ such that $h \cdot \iota_1' = f$ and $h \cdot \iota_2' = g$. Therefore we get a morphism
$$
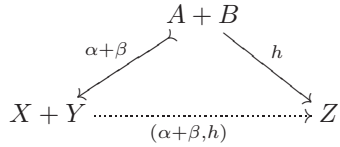
\xymatrix{
  & A+B \ar@{ >->}[dl]_{\alpha+\beta} \ar[dr]^{h} & \\ X+Y \ar@{.>}[rr]_{(\alpha + \beta, h)} & & Z}
$$
in $\Par C$. To prove that
$
\parm {\alpha+\beta} h \cdot \iota_1 = \parm \alpha f
$
in $\Stab C$ it is suffices to look at the following congruence diagram
$$
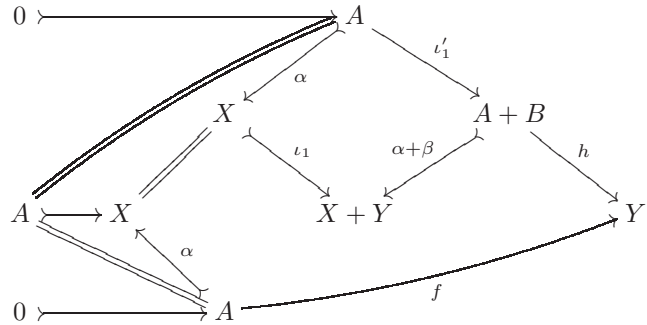
\xymatrix{
0 \ar@{ >->}[rrr]& & & A \ar@{>->}[ld]^{\alpha} \ar[dr]^{\iota_1'} & & & \\
&&  X \ar@{=}[dl]  \ar@{ >->}[dr]^{\iota_1} &  & {A+B} \ar@{ >->}[dl]_{\alpha + \beta}  \ar[dr]^{h} & & \\
A \ar@{=}[drr]_{} \ar@{=}@/^/@<0.5ex>[rruur]^{} \ar@{ >->}[r]_{} & X & & X+Y   && Y\\
0 \ar@{ >->}[rr]_{}  & & A \ar@{ >->}[ul]_{\alpha}  \ar@/_/@<0.3ex>[urrr]_{f} & & & & 
}
$$
Similarly, one checks that $\parm {\alpha+\beta} h \cdot \iota_2 = \parm \beta g$.

To prove the uniqueness of the factorization consider another morphism
$$
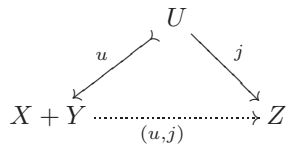
\xymatrix{
& U \ar@{ >->}[dl]_u \ar[dr]^j & \\ X+Y \ar@{.>}[rr]_{(u, j)} & & Z
}
$$
with the same properties as $(\alpha+\beta, h)$, and the congruence diagram
$$
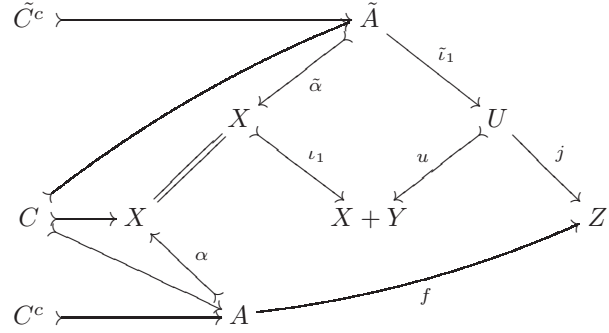
\xymatrix{
\tilde{C^c} \ar@{ >->}[rrr] & & & \tilde{A} \ar@{ >->}[ld]^{\tilde{\alpha}} \ar[dr]^{\tilde{\iota}_1}  & & & \\
& & X \ar@{=}[dl] \ar@{>->}[dr]^{\iota_1} &  & U \ar@{>->}[dl]_u \ar[dr]^{j} & & \\
C \ar@{ >->}[drr]_{} \ar@{ >->}@/^/@<0.5ex>[rruur]^{} \ar@{ >->}[r]_{} & X & & X+Y & & Z\\
C^c \ar@{ >->}[rr]_{}  & & A \ar@{ >->}[ul]_{\alpha}  \ar@/_/@<0.3ex>[urrr]_{f} &  & & & 
}
$$
where the upper square is a pullback, $C+\tilde{C^c}=\tilde{A}$ and $C+C^c=A$.
Symmetrically, we also have the congruence diagram 
$$
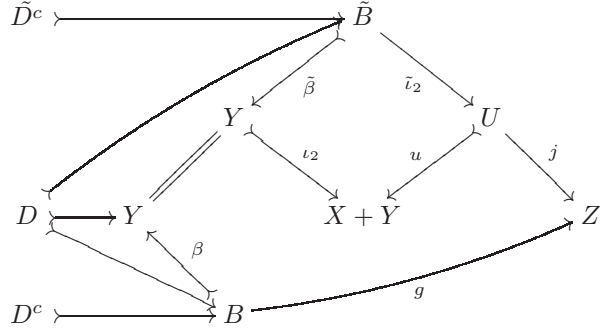
\xymatrix{
\tilde{D^c} \ar@{ >->}[rrr] & & & \tilde{B} \ar@{ >->}[ld]^{\tilde{\beta}} \ar[dr]_{\tilde{\iota}_2}  & & & \\
& & Y \ar@{=}[dl] \ar@{>->}[dr]^{\iota_2} &  & U \ar@{>->}[dl]_u \ar[dr]^j & & \\
D \ar@{ >->}[drr]_{} \ar@{ >->}@/^/@<0.5ex>[rruur]^{} \ar@{ >->}[r]_{} & Y & & X+Y & & Z\\
D^c \ar@{ >->}[rr]_{}  & & B \ar@{ >->}[ul]_{\beta}  \ar@/_/@<0.3ex>[urrr]_g &  & & & 
}
$$
Observe that the universality of coproducts implies that $U= \tilde{A}+\tilde{B}$ as we see in the following diagram where the two squares are pullbacks:
$$
\xymatrix{
\tilde{A} \ar[r]^{\tilde{\iota}_1} \ar[d]_{\tilde{\alpha}} & U \ar[d]^u & \tilde{B} \ar[l]_{\tilde{\iota}_2} \ar[d]^{\tilde{\beta}} \\
X \ar[r]_-{\iota_1} & X+Y & Y. \ar[l]^-{\iota_2}
}
$$
To prove the uniqueness of the factorization we now want to prove that the diagram
$$
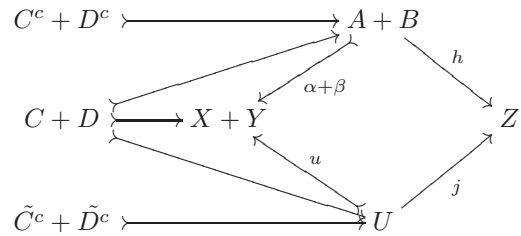
\xymatrix{
C^c+ D^c \ar@{ >->}[rr] & & A+B  \ar@{ >->}[ld]^{\alpha + \beta} \ar[dr]^h  & & \\
C+D \ar@{ >->}[drr] \ar@{ >->}[rru] \ar@{ >->}[r] & X+Y & & Z \\
\tilde{C^c}+\tilde{D^c} \ar@{ >->}[rr] & & U \ar@{ >->}[ul]_u \ar[ru]_j & &
} 
$$
is a congruence diagram.
First observe that $A+B=(C+D)+(C^c+D^c)$ and $U=(C+D)+(\tilde{C^c}+\tilde{D^c})$. Then, we have that $h$ and $f$ coincide on $A$, therefore also on $C$. Moreover, $j$ and $f$ coincide on $C$, hence $h$ and $j$ coincide on $C$. Similarly, $h$ and $j$ coincide on $D$. 
 On the other hand $h$ and $f$ coincide on $C$ and $f$ is trivial on $C^c$, hence $h$ is trivial on $C^c$. Similarly, one sees that $h$ is trivial on $D^c$, and then $h$ is trivial on $C^c+D^c$ by Property~$(1)$. Finally, $j$ is trivial on $\tilde{C^c}$ and $\tilde{D^c}$, hence also on $\tilde{C^c}+\tilde{D^c}$.
\end{proof}



\begin{prop}\label{kernels_preserved}
The functor $\Sigma\colon \mathbb{C}\to \Stab C$ sends $\Cal Z$-kernels into kernels.
\end{prop}

\begin{proof}
Let $k\colon K\to X$ be the $\Cal Z$-kernel of a morphism $f\colon X \to Y$ in $\mathbb{C}$. Since $fk$ is trivial in $\mathbb{C}$ we have that $\Sigma(fk)=0$ in $\Stab C$ by Remark~\ref{zero}. Now, let
$$
\xymatrix{
& A \ar@{ >->}[dl]_\alpha \ar[dr]^g & \\ N \ar@{.>}[rr]_{(\alpha, g)} & & K
}
$$
be a morphism in $\Par C$ such that $\Sigma(f)\cdot \parm \alpha g= \parm \alpha {fg}$ is the zero morphism in $\Stab C$. Again by Remark~\ref{zero}, $fg$ is a trivial morphism in $\mathbb{C}$ and therefore, by the universal property of $\Cal Z$-kernel, there exists a unique morphism $u\colon A \to K$ such that $k u=g$ in $\mathbb{C}$. It is immediate to check that $\Sigma(k)\cdot\parm \alpha u=\parm \alpha g$. Now, let
$$
\xymatrix{
& B \ar@{ >->}[dl]_\beta \ar[dr]^v & \\ N \ar@{.>}[rr]_{(\beta, v)} & & K
}
$$
be another morphism in $\Par C$ such that $\Sigma(k)\cdot\parm \beta v=\parm \alpha g$ in $\Stab C$. We need to prove that $\parm \alpha u =\parm \beta v$. We have a congruence diagram
$$
\xymatrix{
C_A^c \ar@{ >->}[rr]^{\gamma_A^c} & & A \ar@{ >->}[ld]^{\alpha} \ar[dr]_u  \ar[rrd]^{ku} & & & \\
C \ar@{ >->}[drr]_{\gamma_B} \ar@{ >->}[rru]^{\gamma_A} \ar@{ >->}[r]^\gamma & N & & K \ar[r]^k & X \ar[r]^f & Y\\
C_B^c \ar@{ >->}[rr] _{\gamma_B^c} & & B \ar@{ >->}[ul]_{\beta} \ar[ru]^v \ar[rru]_{kv}& & &
} 
$$
where $ku \gamma_A=kv\gamma_B$ and both $ku\gamma_A^c$ and $kv\gamma_B^c$ are trivial morphisms. Since $k$ is a $\Cal Z$-kernel (and in particular a monomorphism) we have that $u \gamma_A=v\gamma_B$ and both $u\gamma_A^c$ and $v\gamma_B^c$ are trivial as well (see Remark~\ref{magenta_remark}).
\end{proof}

\begin{lem}\label{cokernels_coproducts}
Assume we have a commutative diagram
$$\xymatrix{
A_1 \ar[r]^{g_1}\ar@{>->}[d]_{\alpha_1} & B_1 \ar[r]^{q_1} \ar@{>->}[d]_{\beta_1} & Q_1 \ar@{>->}[d]^{\gamma_1}\\
X \ar[r]^g & Y \ar[r]^q & Q \\
A_2 \ar[r]_{g_2} \ar@{>->}[u]^{\alpha_2} & B_2 \ar[r]_{q_2} \ar@{>->}[u]^{\beta_2} & Q_2 \ar@{>->}[u]_{\gamma_2}
}$$
where all columns are sums (hence, all squares are pullbacks). If $q_i=\Cal Z \mbox{-coker}(g_i)$ for $i=1,2$, then $q=\Cal Z \mbox{-coker}(g)$.
\end{lem}

\begin{proof}
Let $h \colon Y \to W$ be a morphism such that $hg$ is trivial. Then for $i=1,2$, $hg\alpha_i=h \beta_i g_i$ is trivial and thus there is a unique $t_i\colon Q_i \to W$ such that $h \beta_i=t_i q_i$. The morphisms $t_1$ and $t_2$ induce a morphism $t\colon Q \to W$ such that $t q=h$. The uniqueness of the factorization follows from the fact that $q_1$ and $q_2$ are epimorphisms.
\end{proof}

\begin{prop}\label{cokernels_preserved}
Let $\parm \alpha g$ be a morphism in $\Stab C$ represented by a morphism in $\Par C$
$$
\xymatrix{
& A \ar@{ >->}[dl]_\alpha \ar[dr]^g & \\
X \ar@{.>}[rr]_{(\alpha, g)} & & Y
}
$$
and assume that for every complemented subobject $\xymatrix{B\ar@{ >->}[r]^\beta &Y}$ the induced morphism $g^{-1}(B)\to B$ has a $\Cal Z$-cokernel in $\mathbb{C}$. Then the cokernel of $\parm \alpha g$ exists in $\Stab C$ and
$$
\mbox{coker}(\parm \alpha g)= \Sigma(\Cal Z\mbox{-coker}(g)).
$$
\end{prop}

\begin{proof}
Let $q\colon Y \to Q$ be the $\Cal Z$-cokernel of $g$ (which exists by hypothesis). Since $qg$ is trivial in $\mathbb{C}$, we have that $\Sigma(q)\cdot \parm \alpha g = \parm \alpha {qg}=0$ in $\Stab C$ by Remark~\ref{zero}.

Now, let
$$
\xymatrix{
& B' \ar@{ >->}[dl]_{\beta'} \ar[dr]^g & \\
Y \ar@{.>}[rr]_{(\beta', h)} & & W
}
$$
be a morphism in $\Par C$ such that $\parm {\beta'} h \cdot \parm \alpha g=0$. We have $\parm {\beta'} h \cdot \parm \alpha g=\parm {\alpha\alpha'}{hg'}$, where $\alpha'$ and $g'$ are defined by the following pullback diagram
$$
\xymatrix{
A' \ar[r]^{g'} \ar@{ >->}[d]_{\alpha'} & B' \ar@{ >->}[d]^{\beta'}\\
A \ar[r]_{g}& {Y.}
}
$$
In particular, $hg'$ is trivial. Let $q'\colon B' \to Q'$ be the $\Cal Z$-cokernel of $g'$. There exists then a unique $w \colon Q'\to W$ such that $h=wq'$. Applying Lemma~\ref{cokernels_coproducts} we get that $\parm {\gamma'} w$ is a morphism in $\Stab C$ such that $\parm {\gamma'}w \cdot \Sigma(q)=\parm {\beta'} h$.
It remains to prove the uniqueness of the factorization. For this, consider another morphism in $\Par C$
$$
\xymatrix{
& Q'' \ar@{ >->}[dl]_{\gamma''} \ar[dr]^v & \\
Q \ar@{.>}[rr]_{(\gamma'', v)} & & W
}
$$
such that $\parm {\gamma'}w \cdot \Sigma(q)=\parm {\gamma''}v \cdot \Sigma(q)=\parm {\beta'} h$. This means that there is a congruence diagram
\begin{equation}\label{CYQW}
\xymatrix@=40pt{
C' \ar@{ >->}[rr]^{\varepsilon'} &   & B' \ar[r]^{q'} \ar@{>->}[dl]^{\beta'} & Q' \ar@{>->}[dl]^{\gamma'} \ar[dr]^w &   \\
C \ar@{ >->}[r]^{\delta} \ar@{ >->}[drr]_{\delta''} \ar@{ >->}[urr]^{\delta'} & Y \ar[r]^q & Q  &    & W \\ 
C'' \ar@{ >->}[rr]_{\varepsilon''} &   & B'' \ar[r]_{q''} \ar@{>->}[ul]_{\beta''} & Q'' \ar@{>->}[ul]_{\gamma''} \ar[ur]
_v &   
}
\end{equation}
where both the parallelograms are pullbacks and $C'$ (resp. $C''$) is the complement of $C$ in $B'$ (resp. in $B''$). Consider the following commutative diagram
$$
\xymatrix@=40pt{
A_1 \ar@{}[rd]|{(1)} \ar@{>->}[r]^{\alpha_1} \ar[d]_{g_1} & A' \ar@{}[rd]|{(2)} \ar@{>->}[r]^{\alpha'} \ar[d]^{g'} & A \ar@{}[rd]|{(3)} \ar[d]^g & A'' \ar@{}[rd]|{(4)} \ar@{>->}[l]_{\alpha''}  \ar[d]^{g''}& A_2 \ar@{>->}[l]_{\alpha_2} \ar[d]^{g_2}\\ 
C \ar@{}[rd]|{(5)}\ar[d]_{q_1}\ar@{>->}[r]^{\delta'} & B'\ar@{}[rd]|{(6)} \ar@{>->}[r]^{\beta'}  \ar[d]^{q'} & Y \ar@{}[rd]|{(7)} \ar[d]^{q} & B''\ar@{}[rd]|{(8)} \ar@{>->}[l]_{\beta''}  \ar[d]^{q''} & C \ar[d]^{q_2}  \ar@{>->}[l]_{\delta''} \\
Q_1 \ar@{>->}[r]_{\gamma_1} & Q' \ar@{>->}[r]_{\gamma'} & Q & Q'' \ar@{>->}[l]^{\gamma''} & Q_2 \ar@{>->}[l]^{\gamma_2}
}
$$
where
\begin{itemize}[leftmargin=*]
\item
the squares $(2)$, $(6)$ and  $(7)$ are the pullbacks already considered and $q'$ is the $\Cal Z$-cokernel of $g'$;
\item
the squares $(1)$, $(3)$ and $(4)$ are pullbacks by construction, and by Lemma~\ref{cokernels_coproducts}, $q''$ is the $\Cal Z$-cokernel of $g''$;
\item
the arrow $q_i$ is defined as the $\Cal Z$-cokernel of $g_i$ for $i=1,2$ and $\gamma_i$ is the morphism induced by the universal property of $\Cal Z$- cokernels;
\item
the squares $(5)$ and $(8)$ are pullbacks by Lemma~\ref{cokernels_coproducts}, and $Q_1$ (resp. $Q_2$) is a complemented subobject of $Q'$ (resp. $Q''$).
\end{itemize}

Notice that $\beta'\delta'=\beta''\delta''$ and since the rectangles $(1)+(2)$ and $(3)+(4)$ are pullbacks, there is no restriction in assuming that $A_1=A_2$. In particular, we can also assume that $Q_1=Q_2$, $q_1=q_2$ and $\gamma'\gamma_1=\gamma''\gamma_2$. Thus we have the following commutative diagram
$$
\xymatrix@=30pt{
D' \ar@{>->}[rr]^{\zeta'}&  & Q'\ar@{>->}[dl]^{\gamma'} \ar[dr]^w & \\
Q_1=Q_2 \ar@{ >->}[r] \ar@{ >->}[rrd]_{\gamma_2} \ar@{ >->}[rru]^{\gamma_1} & Q & & W\\
D'' \ar@{>->}[rr]_{\zeta''} & & Q'' \ar@{>->}[ul]_{\gamma''} \ar[ur]_v&
}
$$
where $D'$ (resp. $D''$) is the complement of $Q_1=Q_2$ in $Q'$ (resp. in $Q''$).
First observe that we have the equalities $v \gamma_2 q_2= v q'' \delta''= w q' \delta'= w \gamma_1 q_1$ and since $q_1=q_2$ is an epimorphism, we get $v \gamma_2= w \gamma_1$.

It remains to prove that $v$ and $w$ are trivial on $D'$ and $D''$, respectively. We first consider the pullbacks
$$
\xymatrix{
A_1 \ar@{}[rd]|{(1)} \ar@{>->}[r]^{\alpha_1} \ar[d]_{g_1} & A'  \ar[d]^{g'} & \ar@{>->}[l]_{\tilde{\alpha}_1} \tilde{A}_1 \ar[d]^{\tilde{g}_1}\\
C \ar@{}[rd]|{(5)} \ar@{>->}[r]^{\delta'} \ar[d]_{q_1} & B' \ar[d]^{q'} & C' \ar@{>->}[l]_{\varepsilon'}\ar[d]^{\tilde{q}_1}\\
Q_1 \ar@{>->}[r]_{\gamma_1} & Q' & D' \ar@{>->}[l]^{\zeta'}
}
$$
where the squares $(1)$ and $(5)$ are as above and $\tilde{q}_1$ is the $\Cal Z$-cokernel of $\tilde{g}_1$.
We already know from diagram~(\ref{CYQW}) that $g \varepsilon'=w q' \varepsilon'=w \zeta' \tilde{q}_1$ is trivial, and since $\tilde{q}_1$ is a $\Cal Z$-cokernel, we get that $w\zeta'$ is trivial as well. Similarly for $v \zeta''$.
\end{proof}

\begin{cor}\label{corollary_cokernel_preserved}
If $\mathbb{C}$ has all $\Cal Z$-cokernels, then the category $\Stab C$ has all cokernels and the functor $\Sigma \colon \mathbb{C}\to \Stab C$ sends $\Cal Z$-cokernels to cokernels.
\end{cor}

\begin{thm}
Let $(\Cal T, \Cal F)$ be a pretorsion theory in a lextensive category $\mathbb{C}$. Assume that $\Cal Z:=\Cal T\cap \Cal F$ is closed under complemented subobjects and that $\mathbb{C}$ has all $\Cal Z$-kernels and all $\Cal Z$-cokernels. Then the stable category $\Stab C$ is well-defined and the quotient functor
$$
\Sigma\colon \mathbb{C}\to \Stab C
$$
sends $\Cal Z$-kernels to kernels and $\Cal Z$-cokernels to cokernels. In particular, it sends $\Cal Z$-short exact sequences to short exact sequences.
\end{thm}

\begin{proof}
It follows from Lemma~\ref{lemma_coproduct}(a), Proposition~\ref{kernels_preserved} and Corollary~\ref{corollary_cokernel_preserved}.
\end{proof}

{
\section{Universal property for hereditary pretorsion theories}
}
In order to prove the universal property of the functor $\Sigma \colon \mathbb{C}\to \Stab C$ extending the one described in \cite{BCG2} we need to assume that the class $\Cal T$ of torsion objects is closed under complemented subobjects.
 Under this assumption we automatically get that also $\Cal Z$ is closed under complemented subobjects, thanks to Proposition~\ref{F_closed_complemented_subobjects}. Moreover, the arguments used in the proof of Lemma~\ref{Lemma_1}(6) also show that for every complemented subobject $\alpha\colon A \rightarrowtail X$ of an object $X$ in $\mathbb{C}$, the left-side square in the commutative diagram
\begin{equation}\label{diagram_1}
\xymatrix{
T(A) \ar[r]^{\varepsilon_A}\ar[d]_{T(\alpha)} & A \ar[r]^{\eta_A} \ar@{>->}[d]^\alpha & F(A) \ar[d]^{F(\alpha)}\\
T(X) \ar[r]_{\varepsilon_X} & X \ar[r]_{\eta_X} & F(X)
}
\end{equation}
is a pullback. Thus we can consider the commutative diagram
$$
\xymatrix{
T(A) \ar[rr]^{\varepsilon_A}\ar@{ >->}[d]_{T(\alpha)} & & A \ar[rr]^{\eta_A} \ar@{ >->}[d]^\alpha & & F(A) \ar@{ >->}[d]^{\iota_1}\\
T(X) \ar[rr]^{\varepsilon_X} & & X \ar[rr]^-{\eta_A+\eta_{A'}} & & {F(A)+F(A')}\\
T(A') \ar[rr]_{\varepsilon_{A'}}\ar@{ >->}[u]^{T(\alpha')} & & A' \ar[rr]_{\eta_{A'}} \ar@{ >->}[u]_{\alpha'} & & F(A') \ar@{ >->}[u]_{\iota_2}
}
$$
where the three coloumns are sums. From Lemma~\ref{cokernels_coproducts} it follows that $\eta_A+\eta_{A'}$ is the $\Cal Z$-cokernel of $\varepsilon_X$, hence $F(A)+F(A')=F(X)$. In particular, if $A,A' \in \Cal F$, then $A+A'=F(A)+F(A')=F(A+A')$, and therefore $\Cal F$ is closed under coproducts. Thus, we have the following result.

\begin{prop}\label{proposition_TFZ_closed}
Let $(\Cal T, \Cal F)$ be a pretorsion theory in a lextensive category $\mathbb{C}$ and assume that $\Cal T$ is closed under complemented subobjects in $\mathbb{C}$. Then the three subcategories $\Cal T, \Cal F$ and $\Cal Z$ are  closed in $\mathbb{C}$ under complemented subobjects and coproducts.
\end{prop}

We now recall the definition of torsion theory functor introduced in \cite{BCG2}.

\begin{defn}\cite{BCG2}\label{tt-functor}
 Let $({\mathbb A}, {\mathcal T},  {\mathcal F})$ be a category $\mathbb A$ with a given pretorsion theory $({\mathcal T},  {\mathcal F})$ in $\mathbb A$. If $({\mathbb B}, {\mathcal T'},  {\mathcal F'})$ is a pointed category $\mathbb B$ with a given torsion theory $({\mathcal T'},  {\mathcal F'})$ in it, we say that a functor $G \colon \mathbb A \rightarrow \mathbb B$ is a \emph{torsion theory functor} if the following two properties are satisfied:
 \begin{enumerate}
 \item $G(A) \in {\mathcal T}'$ for any $A \in \mathcal T$, $G(B) \in {\mathcal F}'$ for any $B \in \mathcal F$;
 \item if $T(A) \rightarrow A \rightarrow F(A)$ is the canonical short $\mathcal Z$-exact sequence associated with $A$ in the pretorsion theory $({\mathcal T},  {\mathcal F})$, then $$0 \rightarrow G(T(A)) \rightarrow G(A) \rightarrow G(F(A)) \rightarrow 0$$ is a short exact sequence in $\mathbb B$.
 \end{enumerate}
 \end{defn}

\begin{thm}\label{theorem_tt_functor}
Let $(\Cal T, \Cal F)$ be a pretorsion theory in a lextensive category $\mathbb{C}$ and assume that $\Cal T$ is closed under complemented subobjects. Then:
\begin{enumerate}
\item
the stable category $\Stab C$ is well-defined;
\item
the functor $\Sigma\colon \mathbb{C}\to \Stab C$ sends $\Cal Z$-kernels to kernels and $\Cal Z$-short exact sequences to short exact sequences. It also sends $\Cal Z$-cokernels to cokernels provided all $\Cal Z$-cokernels exist in $\mathbb{C}$.
\item
$(\Cal T, \Cal F)$ is a torsion theory in $\Stab C$ and $\Sigma\colon \mathbb{C}\to \Stab C$ is a torsion theory functor.
\end{enumerate}
\end{thm}

\begin{proof} $(1)$ From Proposition~\ref{proposition_TFZ_closed} the three subcategories $\Cal T, \Cal F$ and $\Cal Z$ are all closed under complemented subobjects and coproducts. By Lemma~\ref{lemma_coproduct}(b) the assumptions at the beginning of Section~\ref{section_STABLE} are satisfied, thus $\Stab C$ is well-defined.

$(2)$ Consider a short $\Cal Z$-exact sequence 
$$
\xymatrix{T(X)\ar[r]^-{\varepsilon_X}& X \ar[r]^-{\eta_X} & F(X)}
$$
and observe that since the left-side square in the diagram~(\ref{diagram_1}) is a pullback, the assumption of Proposition~\ref{cokernels_preserved} are satisfied for the morphism $\Sigma(\varepsilon_X)$. Thus the statement follows from Proposition~\ref{kernels_preserved} and Proposition~\ref{cokernels_preserved}.

$(3)$ Consider a morphism $\parm \alpha f$ from $X \in \Cal T$ to $Y \in \Cal F$,
$$
\xymatrix{
& A \ar@{>->}[dl]_{\alpha} \ar[dr]^f &\\
X \ar@{.>}[rr]_{\parm \alpha f } & & Y
}
$$
Since $\Cal T$ is closed under complemented subobjects, $A \in \Cal T$ and $f$ is a trivial morphism in $\mathbb{C}$. Thus, $\parm \alpha f=0$ by Remark~\ref{zero}.

Now, for any object $X \in \Stab C$ there is a short $\Cal Z$-exact sequence in $\mathbb{C}$	of the form
$$
\xymatrix{T(X)\ar[r]^-{\varepsilon_X}& X \ar[r]^-{\eta_X} & F(X)}
$$
and by $(2)$ its image via the functor $\Sigma$ is a short exact sequence in $\Stab C$.
\end{proof}

\begin{thm}\label{universal}
Let $(\Cal T, \Cal F)$ be a pretorsion theory in a lextensive category $\mathbb{C}$ and assume that $\Cal T$ is closed under complemented subobjects. The quotient functor $\Sigma \colon \mathbb C \to \Stab C$  is universal among all finite coproduct preserving \emph{torsion theory functors} $G \colon \mathbb{C} \to \mathbb X$, where $\mathbb X$ has a torsion theory $(\Cal T', \Cal F')$ and finite coproducts. This means that any finite coproduct preserving torsion theory functor $G \colon \mathbb{C} \to \mathbb X$ factors uniquely through $\Sigma$:
$$
\xymatrix{
\mathbb{C} \ar[rr]^-{\Sigma} \ar[rd]_{\forall G} & & \Stab C \ar@{.>}[ld]^{\exists ! H} \\
& \mathbb X &
}
$$
Moreover, the induced functor $H$ preserves finite coproducts, and it is a torsion theory functor.
\end{thm}

\begin{proof}
Since $\mathbb{C}$ and $\mathsf{Stab}(\mathbb C)$ have the same objects it is clear that the definition of the functor $H$ on the objects is entirely determined by the functor $G$, so that
$H(X) = G(X)$, for any object $X$ in $\Stab C$. Let then $\parm \alpha f  \colon X \rightarrow Y$ be a morphism in $\Stab C$, represented by
$$
\xymatrix{
& A \ar@{ >->}[dl]_\alpha \ar[dr]^f & \\
X \ar@{.>}[rr]_{(\alpha, f)} & & Y
}
$$
and observe that it is equal to $f+0$, the morphism induced by the universal property of the coproduct $A \coprod A^c = X$ (see \cite[Proposition 1.5]{BCG2}).
Again, the condition $H \cdot \Sigma= G$ and the fact that $H$ has to preserve binary coproducts force the definition of the functor $H$ on morphisms:
$$
H(\parm \alpha f) = G(f)+0.
$$
The above arguments already prove the uniqueness of the functor $H$ with the above properties.
The proof of the fact that $H$ is well-defined on morphisms is the same as the one given in \cite[Theorem~2.3]{BCG2} for the case of internal preorders in a pretopos.
\end{proof}

\section{Examples}\label{section_examples}

In this section we are going to present some examples of pretorsion theories in lextensive categories that satisfy the assumptions of Theorems~\ref{theorem_tt_functor} and~\ref{universal}.

\subsection*{Internal preorders in a pretopos}
As already said, this work is inspired by the results obtained in \cite{BCG, BCG2} for the category $\mathsf{PreOrd}(\mathbb C)$ of internal preorders in a pretopos $\mathbb C$ (recall that a pretopos is a (Barr-)exact and extensive category \cite{Elep}). An object $(A, \rho)$ in $\mathsf{PreOrd}(\mathbb C)$ is a relation $\langle r_1,r_2 \rangle \colon \rho \rightarrow A \times A$ on $A$, i.e. a subobject of $A \times A$, that is \emph{reflexive}, i.e. it contains the ``discrete relation'' $\langle 1_A, 1_A \rangle \colon A \rightarrow A \times A$ on $A$, denoted by $\Delta_A$, 
and \emph{transitive}: there is a
morphism $\tau \colon \rho \times_A \rho \rightarrow \rho$ such that $r_1 \tau =  r_1 p_1$ and $r_2 \tau =  r_2 p_2$, where $(\rho \times_A \rho, p_1, p_2)$ is the pullback
$$
\xymatrix{ \rho \times_A \rho \ar[r]^-{p_2} \ar[d]_-{p_1}& \rho \ar[d]^{r_1} \\
\rho \ar[r]_{r_2} & A.
}
$$
A morphism $(A, \rho) \rightarrow  (B, \sigma)$ in the category $\mathsf{PreOrd}(\mathbb C)$ of preorders in $\mathbb C$ is a pair $(f,\hat{f})$ of morphisms in $\mathbb C$ making the diagram
$$
\xymatrix{  \rho \ar@<.5ex>[d]^{r_2} \ar@<-.5ex>[d]_{r_1}  \ar[r]^{\hat{f}}  & \sigma \ar@<.5ex>[d]^{s_2} \ar@<-.5ex>[d]_{s_1}  \\
 A \ar[r]_{f} & {B,} &
}
$$
commute, i.e. $f r_1= s_1 \hat{f} $ and $f r_2= s_2 \hat{f}$.

There is a pretorsion theory $(\mathsf{Eq}(\mathbb C), \mathsf{Par}(\mathbb C))$ in $\mathsf{PreOrd}(\mathbb C)$, where $\Equiv C$ is the full subcategory of all symmetric preorders (i.e. of all equivalence relations) while $\ParOrd C$ is the full subcategory of all antisymmetric preorders (i.e. of all partial orders) \cite{FFG2}. The class $\Disc C$ of trivial objects consists of all the discrete relations, i.e. those preorders of the form $(A,\Delta_A)$.

\subsection*{Two pretorsion theories in the category of topological spaces}
 The next two examples are fully described in \cite{FFG2}. Recall that a topology on a set $X$ is called a partition topology if there exists a partition of $X$ which is a base of open sets of the topology. Let $\Cal P$ be the full subcategory of the category $\mathsf{Top}$ of topological spaces whose objects are all the spaces whose topology is a partition topology and let $\Cal K$ be the full subcategory of all Kolmogorov spaces (i.e.$\mathsf{T_0}$ spaces). Then $(\Cal P, \Cal K)$ is a pretorsion theory in $\mathsf{Top}$ \cite[Section~6.6]{FFG}. The class $\Cal Z$ of trivial objects consists of all the discrete topological spaces.
Another pretorsion theory in $\mathsf{Top}$ is given by the pair $(\Cal T, \Cal F)$, where $\Cal T$ is the subcategory of all topological spaces for which the connected components are open subsets and $\Cal F$ is the subcategory of all totally disconnected topological spaces \cite[Section~6.7]{FFG2}. Also in this case, the class $\Cal Z$ of trivial objects consists of the discrete topological spaces.

\subsection*{The category of endomappings of finite sets}

In \cite{FH}, the authors describe a pretorsion theory in the category $\Cal M$ of endomappings of finite sets. The objects of $\Cal M$ are all pairs $(X,f)$, where $X=\{1,2,\dots,n\}$ for some $n\geq 0$ is a finite set and $f\colon X \to X$ is a map. A morphism $\varphi\colon (X,f)\to (Y,g)$ in $\Cal M$ is a map $\varphi\colon X\to Y$ such that $g\cdot \varphi=\varphi\cdot f$.
There is a pretorsion theory $(\Cal C, \Cal F)$ on $\Cal M$ \cite[Theorem~7.2]{FH}, where $\Cal C$ is the full subcategory of $\Cal M$ whose objects are the pairs $(X,f)$ with $f$ a bijection and $\Cal F$ is the full subcategory of $\Cal M$ whose objects are the pairs $(X,f)$ where $f^{|X|}=f^{|X|+1}$ (here $\Cal C$ and $\Cal F$ stand for ``cycles" and ``forests" respectively, according to the description of the objects $(X,f)$ in terms of the oriented graph $\Cal G_f$ associated with $f$ \cite[Section~6]{FH}). The trivial objects are those of the form $(X,\Id X)$.

The category $\mathcal{M}$ can be embedded in $\mathsf{PreOrd}(\mathsf{Set})$ via the assignment $(X,f)\mapsto (X,\rho_f)$, where $\rho_f$ is the preorder on $X$ defined by $x\rho_f y$ if and only if $x=f^t(y)$ for some integer $t\geq0$. Accordingly, $\Cal M$ can be identified with a (non-full) subcategory of $\mathsf{PreOrd}(\mathsf{Set})$ and under this identification we have that $\Cal C=\Cal M \cap \mathsf{Equiv}(\mathsf{Set})$ and $\Cal F= \Cal M \cap \mathsf{ParOrd}(\mathsf{Set})$.

The category $\Cal M$ can be also seen as a full subcategory of the category of oriented graphs with loops but no multiple edges between two vertices. An object $(X, f)$ corresponds to a graph $(X, \Cal G_f)$, where the set of edges is $\Cal G_f=\{(x, f(x))\mid x \in X\}$ and thus the objects of $\Cal M$ are the oriented graphs in which every vertex has outer degree equal to one.

The complemented subobjects of an object $(X,f)$ can be easily described in terms of the oriented graph $\Cal G_f$ associated with $f$: they correspond to (unions of) the connected components of $\Cal G_f$.
The category $\Cal M$ is lextensive, and the sum of two objects $(X,f),(Y,g)\in \Cal M$ is computed ``componentwise": $(X,f)+(Y,g)=(X+Y,f+g)$.




It is worth mentioning that in \cite{FH} the authors propose a construction of the ``stable category", getting a slightly different result from ours. Here we give an explicit description of our construction by using the same notations as in \cite{FH}, so that the reader can easily compare the two approaches. First, we embed $\Cal M$ in a pointed category $\Cal M_*$ whose objects are the pairs $(X_*,f_*)$ where $X_*=\{0,1,\dots,n\}$ is a finite set and $f_*\colon X_*\to X_*$ is mapping such that $f_*^{-1}(0)=\{0\}$. A morphism $\varphi_*\colon (X_*,f_*)\to (Y_*,g_*)$ between two objects of $\Cal M_*$ is a map $\varphi_*\colon X_*\to Y_*$ such that $g_* \cdot \varphi_*=\varphi_*\cdot f_*$ and $\varphi_*(0)=0$. Thus $\Cal M$ can be embedded in $\Cal M_*$ via the assignment $(X,f)\mapsto (X_*,f_*)$, where $X_*:=X\cup\{0\}$ and $f_*$ is the extension of $f$ with $f_*(0)=0$. Under this identification we also have $\Ob (\Cal M)=\Ob (\Cal M_*)$ and we can write $X=X_*\setminus\{0\}$. The category $\Cal M_*$ is equivalent to the category of partial maps $\Par {\Cal M}$. A morphism $\varphi_*\colon (X_*,f_*)\to (Y_*,g_*)$ in $\Cal M_*$ corresponds to a partial morphism 
$$
\xymatrix{
 & (\varphi_*^{-1}(Y), f') \ar@{ >->}[dl]_{\iota} \ar[dr]^{\varphi'} & \\
 (X, f) \ar@{.>}[rr]_{\varphi_*} & & (Y,g)
}
$$
where $f'$ and $\varphi'$ are the restrictions of $f$ and $\varphi_*$ at $\varphi_*^{-1}(Y)$ and $\iota$ is the inclusion map (the idea is that ``we can forget all the elements of $X_*$ going to $0$"). Thus, the congruence of Proposition~\ref{cong} can be rephrased as follows: given two morphisms $\varphi_*,\psi_*\colon (X_*,f_*)\to(Y_*,g_*)$ in $\Cal M_*$ we have $\varphi_*\sim \psi_*$ if there exists a complemented subobject $(C,f')$ of $(X, f)$ such that $\varphi_*$ coincides with $\psi_*$ on $C_*$ and both $\varphi_*$ and $\psi_*$ are trivial on the complement of $C$ in $X$.

\subsection*{The category of all (small) categories}

Consider the category $\mathsf{Cat}$ of all small categories. According to \cite{Xarez}, a category $\mathbb{T} \in \mathsf{Cat}$ is a \emph{symmetric} category if for every $X,Y \in \mathbb{T}$, if $\Hom(X,Y)\neq \emptyset$, then $\Hom(Y,X)\neq \emptyset$. A category $\mathbb{F}\in \mathsf{Cat}$ is \emph{antisymmetric} if for every $X,Y \in \mathbb{F}$, if $\Hom(X,Y)\neq \emptyset$ and $\Hom(Y,X)\neq \emptyset$, then $X=Y$. There is a pretorsion theory in $\mathsf{Cat}$ given by the pair $(\mathsf{CatEquiv},\mathsf{CatOrd})$ where $\mathsf{CatEquiv}$ and $\mathsf{CatOrd}$ are the full subcategories of $\mathsf{Cat}$ of symmetric and antisymmetric categories respectively \cite[Theorem~5.1]{Xarez}. In this example the trivial objects are classes of monoids. According to our results, also this pretorsion theory then admits a universal stable category, thanks to Theorem \ref{universal}, a fact that was not considered in \cite{Xarez}.

\subsection*{Another pretorsion theory in $\mathsf{Cat}$} This last example can be found in \cite{BCGT}. It is possible to consider another pretorsion theory in $\mathsf{Cat}$ having the groupoids as torsion objects and the skeletal categories (i.e. the ones where every isomorphism is an automorphism) as torsion-free objects. In this case, the trivial objects are classes of groups. Since the subcategory of groupoids is closed under complemented subobjects in $\mathsf{Cat}$, Theorem~\ref{universal} can be applied also in this case, and it is then possible to associate a stable category with this pretorsion theory in $\mathsf{Cat}$.

\begin{rem}
We would like to conclude the paper by mentioning the existence of a wide literature on radical and connectedness theories, that are clearly related to the present work. In particular the book of B.J. Gardner and R. Wiegandt \cite{Gar} and the article by M.M. Clementino and W. Tholen \cite{CT} have a rich bibliography on radical theories and connectedness theories, respectively. A categorical approach to radical and torsion theories based on the notion of ideal of morphisms \cite{Ehr} has been recently developed by M. Grandis, G. Janelidze, L. M\'arki \cite{GJM, GJ} and S. Mantovani \cite{Man}.
\end{rem}

\newpage

\begin{thebibliography}{A}

\bibitem{Barr} M. Barr, P. A. Grillet, and D. H. Van Osdol, {\em Exact categories and categories of sheaves}, in: Lect. Notes Math. \textbf{236} (1971), Springer, Heidelberg,  1--120.

\bibitem{BCG} F. Borceux, F. Campanini and M. Gran, {\em The stable category of internal prorders in a pretopos I: general theory}, J. Pure Applied Algebra \textbf{226} (2022) 106997.

\bibitem{BCG2} F. Borceux, F. Campanini and M. Gran, {\em The stable category of internal preorders in a pretopos II: the universal property}, Ann. Mat. Pura  Appl., accepted for publication.


\bibitem{BCGT} F. Borceux, F. Campanini, M. Gran and W. Tholen, {\em Groupoids and skeletal categories form a pretorsion theory in $\mathsf{Cat}$}, preprint, arXiv:2207.08487, (2022).


\bibitem{BG} D. Bourn, M. Gran, {\em Torsion theories in homological categories}, J. Algebra \textbf{305} (2006)  pp. 18--47.

\bibitem{CLW} A. Carboni, S. Lack and R.F.C. Walters, {\em Introduction to extensive and distributive categories}, J. Pure Appl. Algebra \textbf{84} (1993) pp. 145--158.

\bibitem{CDT} M.M. Clementino, D. Dikranjan and W. Tholen, {\em Torsion theories and radicals in normal categories}, J. Algebra \textbf{305} (2006)  pp. 98--129.

\bibitem{CT} M.M. Clementino and W. Tholen, {\em Separation versus connnectedness}, Topology Appl. \textbf{75} (2) (1997) 143-181.

\bibitem{D} S.C. Dickson, {\em A torsion theory for abelian categories}, Trans. Amer. Math. Society \textbf{21} (1966) pp. 223--235.

\bibitem{Ehr} C. Ehresmann, {\em Sur une notion g\'en\'erale de cohomologie}, C. R. Acad. Sci. Paris \textbf{259} (1964) pp. 2050--2053.

\bibitem{FF} A. Facchini and C.A. Finocchiaro, {\em Pretorsion theories, stable category and preordered sets}, Ann. Mat. Pura  Appl. (4) \textbf{199} no. 3 (2020) pp.  1073--1089.

\bibitem{FFG} A. Facchini, C.A. Finocchiaro and M. Gran, {\em Pretorsion theories in general categories}, J. Pure Appl. Algebra \textbf{225} (2) (2021) 106503.

\bibitem{FFG2} A. Facchini, C.A. Finocchiaro and M. Gran, {\em  A new Galois structure in the category of internal preorders}, Theory Appl. Categories \textbf{35} (2020) pp. 326--349.

\bibitem{FH} A. Facchini and L Heidari Zadeh, {\em An extension of properties of symmetric group to monoids and a pretorsion theory on a category of mappings}, J. Algebra Appl. \textbf{18} (12) (2019) 1950234.

\bibitem{Gar} B.J. Gardner, and R. Wiegandt, Radical theory of rings, Monographs and Textbooks in Pure and Applied Mathematics, \textbf{261}, Marcel Dekker, Inc., New York, (2004).

\bibitem{GM} M. Gran and A. Michel, {\em Torsion theories and coverings of preordered groups}, Algebra Universalis  \textbf{82}, 22 (2021).

\bibitem{GJ} M. Grandis and G. Janelidze, {\em From torsion theories to closure operators and  factorization systems,} Categories and General Algebraic Structures with Applications \textbf{12} (1) (2019) pp. 89--121.

\bibitem{GJM} M. Grandis, G. Janelidze and L. M\'arki, {\em Non-pointed exactness, radicals, closure operators}, J. Aust. Math. Soc. \textbf{94} (2013) 348--361.

\bibitem{JT} G. Janelidze and W. Tholen, {\em Characterization of torsion theories in general categories}, in  ``Categories in algebra, geometry and mathematical physics'', A. Davydov, M. Batanin, M. Johnson, S. Lack and A. Neeman Eds., Contemp. Math. {431}, Amer. Math. Soc., Providence, RI, 2007, pp. 249--256. 

\bibitem{Elep} P.T. Johnstone, ``Sketches of an elephant: a topos theory compendium'', Vol. 1, Oxford Logic Guides 43, Oxford Univ. Press, New York 2002.

\bibitem{Man} S. Mantovani, {\em Torsion theories for crossed modules}, Invited talk at the Workshop on Category Theory and Topology, Université catholique de Louvain, September 2015.

\bibitem{Michel} A. Michel, {\em Torsion theories and coverings of $V$-groups}, Appl. Categ. Struct. \textbf{24} (2022) https://doi.org/10.1007/s10485-021-09670-w.

\bibitem{RR} E. Robinson and G. Rosolini, {\em Categories of partial maps}, Information and Computation \textbf{79} (1988) 95--130.

\bibitem{Xarez} J. Xarez, {A pretorsion theory for the category of all categories}, Cahiers Top. G\'eom. Diff. Cat\'eg.  \textbf{LXIII}  (1) (2021), 25-34.

\end{thebibliography}

\end{document}